\documentclass[12pt]{amsart}

\usepackage[margin=2.5cm]{geometry}
\usepackage{a4wide}

\newtheorem{theorem}{Theorem}
\theoremstyle{remark}
\newtheorem{remark}[theorem]{Remark}

\title{A note on the polynomial-like iterative equations order}
\author{Szymon Draga}
\address{Institute of Mathematics\\ University of Silesia\\ Bankowa 14\\ 40-007 Katowice\\ Poland}
\email{szymon.draga@gmail.com}

\subjclass[2010]{39B12}
\keywords{Continuous solution, Iterate, Polynomial-like iterative equation, Recurrence relation}

\newcommand{\N}{\mathbb N}
\newcommand{\Root}{\mathcal R}
\newcommand{\R}{\mathbb R}
\newcommand{\Z}{\mathbb Z}
\newcommand*\xbar[1]{\hbox{\vbox{\hrule height 0.5pt \kern0.5ex\hbox{\kern-0.1em\ensuremath{#1}\kern-0.1em}}}}

\begin{document}

\begin{abstract}
We show that, under reasonable assumptions, two negative roots can be eliminated from the characteristic equation of a polynomial-like iterative equation. This result gives a new case where we may lower the order of such an equation.
\end{abstract}

\maketitle

\section{Introduction}
Let $n$ be a positive integer and $I\subset\R$ be an interval. We are interested in so-called polynomial-like iterative equations, namely, equations of the form
\begin{equation}\label{eqmain}
a_nf^n(x)+\ldots+a_1f(x)+a_0x=0,
\end{equation}
where $f^k$ stands for the $k$-fold iterate of a self-mapping unknown function $f\colon I\to I$, the coefficients $a_n,\ldots,a_1,a_0$ are given real numbers and $a_0\neq0$. In general, it is difficult to find all continuous functions satisfying equation \eqref{eqmain} even in the case $n=3$; a partial solution in this case was given in \cite{zhang-gong} and the complete solution for $n=2$ was presented in \cite{matkowski-zhang1,nabeya}. One of methods for finding solutions to equation \eqref{eqmain} involves lowering its order. Such results were obtained in \cite{draga-morawiec,yang-zhang,zhang-zhang} (see Theorem \ref{thmref} below); the present paper contains a new result in this spirit. For a~similar investigation concerning non-homogenous polynomial-like iterative equations, where zero on the right-hand side is replaced by an arbitrary continuous function, see, e.g. \cite{zhang-xu-zhang}.

We shall recall some basic properties of solutions to polynomial-like equations. Assume that a~continuous function $f\colon I\to I$ satisfies equation \eqref{eqmain}. It can be easily shown that $f$ is injective (see, e.g. \cite[Lemma 2.1]{draga-morawiec}) and therefore monotone. Assuming that $f(x)=rx$ we obtain the so-called characteristic equation of \eqref{eqmain}:
\begin{equation}\label{eqchar}
a_nr^n+\ldots+a_1r+a_0=0.
\end{equation}
This equation may also be considered as the characteristic equation of the recurrence relation
\begin{equation}\label{eqrec}
a_nx_{j+n}+\ldots+a_1x_{j+1}+a_0x_j=0
\end{equation}
in which the sequence $(x_j)_{j\in\N_0}$ is obtained in the following way: We choose $x_0\in I$ arbitrarily and put $x_j=f(x_{j-1})$ for $j\in\N$. It is easy to see that $f$ satisfies \eqref{eqmain} if and only if $(x_j)_{j\in\N_0}$ satisfies \eqref{eqrec}.

Since the function $f$ is monotone, the sequence $(x_j)_{j\in\N_0}$ is either monotone (in the case of increasing $f$) or anti-monotone (in the case of decreasing $f$). By \textit{anti-monotone} we mean that the expression $(-1)^j(x_{j+1}-x_j)$ does not change its sign when $j$ runs through $\N_0$. Consider $y_0\in I$ and define a sequence $(y_j)_{j\in\N_0}$ in the same way as we did for $x_0$. Similarly, the sequence $(x_j-y_j)_{j\in\N_0}$ has a~constant sign, in the case of increasing $f$, or alternates in sign, in the case of decreasing $f$. Let us note that the sequence $(x_{j+1}-y_j)_{j\in\N_0}$ also has the same property.

In the case where $f$ is surjective (and hence bijective) we can consider the dual equation
\begin{equation}\label{eqdual}
a_0f^n(x)+\ldots+a_{n-1}f(x)+a_nx=0.
\end{equation}
Putting $f^{-n}(x)$ in place of $x$ we see that $f$ satisfies \eqref{eqmain} if and only if $f^{-1}$ satisfies \eqref{eqdual}. We can also extend the above defined sequence $(x_j)_{j\in\N_0}$ to the whole $\Z$ by setting $x_{-j}=f^{-1}(x_{-j+1})$ for $j\in\N$. Then relation \eqref{eqrec} is satisfied for all $j\in\Z$.

For the theory of linear recurrence relations we refer the reader, for instance, to \cite[\S 3.2]{jerri}. We shall recall only the most significant theorem in this matter. In order to do this and simplify the writing we introduce the following notation: For a~given polynomial $c_nr^n+\ldots+c_1r+c_0$ we denote by $\Root(c_n,\ldots,c_0)$ the collection $\{(r_1,k_1),\ldots,(r_p,k_p)\}$ of all pairs of pairwise distinct (complex) roots $r_1,\ldots,r_p$ and their multiplicities $k_1,\ldots,k_p$, respectively. Here and throughout the present paper by a~polynomial we mean a polynomial with real coefficients. Note that in the introduced notation $k_1+\ldots+k_p$ equals the degree of $c_nr^n+\ldots+c_1r+c_0$ and by writing $(\mu,k),(\xbar{\mu},k)\in\Root(c_n,\ldots,c_0)$ we mean $\mu$ to be non-real.

\begin{theorem}\label{thmrec}
Assume that
\[ \Root(a_n,\ldots,a_0)=\{(\lambda_1,l_1),\ldots,(\lambda_p,l_p),(\mu_1,m_1),(\xbar{\mu_1},m_1),\ldots,(\mu_q,m_q),(\xbar{\mu_q},m_q)\}. \]
Then a real-valued sequence $(x_j)_{j\in\N_0}$ is a solution to \eqref{eqrec} if and only if it is given by
\[ x_j=\sum_{k=1}^pA_k(j)\lambda_k^j+\sum_{k=1}^q\big(B_k(j)\cos j\phi_k+C_k(j)\sin j\phi_k\big)|\mu_k|^j\quad\mbox{ for }j\in\N_0, \]
where $A_k$ is a polynomial whose degree equals at most $l_k-1$ for $k=1,\ldots,p$ and $B_k,C_k$ are polynomials whose degrees equal at most $m_k-1$, with $\phi_k$ being an argument of $\mu_k$, for $k=1,\ldots,q$.
\end{theorem}

It is worth mentioning that the above theorem is also valid for sequences defined on the whole $\Z$. We shall use this fact in the proof of our main result.

\section{The main result}
It was observed by Matkowski and Zhang in \cite{matkowski-zhang2} that if a~polynomial $b_mr^m+\ldots+b_1r+b_0$ divides $a_nr^n+\ldots+a_1r+a_0$ and $f$ satisfies
\begin{equation}\label{eqreduced}
b_mf^m(x)+\ldots+b_1f(x)+b_0x=0,
\end{equation}
then $f$ satisfies also \eqref{eqmain}. One of methods for solving equation \eqref{eqmain} involves a~partial converse of this statement. More precisely, we want to find a divisor of the polynomial $a_nr^n+\ldots+a_1r+a_0$ such that the corresponding polynomial-like iterative equation of lower order is satisfied. Known results, concerning elimination of non-real roots or real roots of opposite sign, are listed below.

\begin{theorem}\label{thmref}
\emph{\textbf{(i)} \cite[Thm. 3.3]{draga-morawiec} (cf. \cite[Thm. 5]{yang-zhang} and \cite[Thm. 1]{zhang-zhang})} Assume that
\[ \Root(a_n,\ldots,a_0)=\{(\lambda_1,l_1)\ldots,(\lambda_p,l_p),(\mu_1,k_1),(\xbar{\mu_1},k_1),\ldots,(\mu_q,k_q),(\xbar{\mu_q},k_q)\}. \]
If $|\lambda_1|\le\ldots\le|\lambda_p|<|\mu_1|\le\ldots\le|\mu_q|$, then a continuous function $f\colon I\to I$ satisfies equation \eqref{eqmain} if and only if it satisfies \eqref{eqreduced} with
\[ \Root(b_m,\ldots,b_0)=\{(\lambda_1,l_1)\ldots,(\lambda_p,l_p)\}. \]
\emph{\textbf{(ii)} \cite[Thms. 4.1 and 4.2]{draga-morawiec}} Assume that
\[ \Root(a_n,\ldots,a_0)=\{(r_1,k_1),(r_2,k_2),(\lambda_1,l_1),\ldots,(\lambda_p,l_p)\}. \]
Let also $|r_1|<|\lambda_1|\le\ldots\le|\lambda_p|<|r_2|$ and $r_1$, $r_2$ be real with $r_1r_2<0$; say $r_i>0$ and $r_j<0$. Then a continuous increasing surjection $f\colon I\to I$ satisfies equation \eqref{eqmain} if and only if it satisfies \eqref{eqreduced} with
\[ \Root(b_m,\ldots,b_0)=\{(r_i,k_i),(\lambda_1,l_1),\ldots,(\lambda_p,l_p)\}. \]
If $r_i\neq1$, then a continuous decreasing surjection $f\colon I\to I$ satisfies equation \eqref{eqmain} if and only if it satisfies \eqref{eqreduced} with
\[ \Root(b_m,\ldots,b_0)=\{(r_j,k_j),(\lambda_1,l_1),\ldots,(\lambda_p,l_p)\}. \]
If $r_i=1$, then a continuous decreasing surjection $f\colon I\to I$ satisfies equation \eqref{eqmain} if and only if it satisfies \eqref{eqreduced} with
\[ \Root(b_m,\ldots,b_0)=\{(1,1),(r_j,k_j),(\lambda_1,l_1),\ldots,(\lambda_p,l_p)\}. \]
\end{theorem}

Those results were proved by examining the asymptotic behaviour of the sequence of consecutive iterates of the unknown function at a~given point. Using a~similar approach we obtain our new result, concerning elimination of negative roots, which reads as follows.

\begin{theorem}\label{thmlower}
Assume that
\[ \Root(a_n,\ldots,a_0)=\{(r_1,k_1),(r_2,k_2),(\lambda_1,l_1),\ldots,(\lambda_p,l_p)\}. \]
If $|r_2|<|\lambda_1|\le\ldots\le|\lambda_p|<|r_1|$ and $r_1$, $r_2$ are real with $r_1<-1<r_2<0$, then a~continuous surjection $f\colon I\to I$ satisfies equation \eqref{eqmain} if and only if it satisfies equation \eqref{eqreduced} with
\[ \Root(b_m,\ldots,b_0)=\{(r_i,k_i),(\lambda_1,l_1),\ldots,(\lambda_p,l_p)\}, \]
where $i=1$ or $i=2$.
\end{theorem}

\begin{proof}
Choose $x\in I$ arbitrarily. Define a sequence $(x_j)_{j\in\Z}$ by putting $x_0=x$, $x_j=f(x_{j-1})$ and $x_{-j}=f^{-1}(x_{-j+1})$ for $j\in\N$. Then relation \eqref{eqrec} is satisfied for all $j\in\Z$. Therefore, by Theorem \ref{thmrec}, we have
\[ x_j=A(j)r_1^j+F(j)+B(j)r_2^j\quad\mbox{ for }j\in\Z, \]
where $A$, $B$ are polynomials and $F$ stands for the part of solution to \eqref{eqrec} for which the roots $\lambda_1,\ldots,\lambda_p$ are responsible. We shall show that either $A\equiv0$ or $B\equiv0$.

For an indirect proof suppose that both polynomials $A$ and $B$ are non-zero. Denote by $s$ and $t$ degrees of $A$ and $B$, respectively. Similarly, let $a$ and $b$ be the leading coefficients of $A$ and $B$.

Since
\begin{multline*}
(-1)^j(x_{j+1}-x_j)=\big(A(j+1)r_1-A(j)\big)|r_1|^j\\ +(-1)^j\big(F(j+1)-F(j)\big)+\big(B(j+1)r_2-B(j)\big)|r_2|^j,
\end{multline*}
we have
\begin{align}
& \lim_{j\to-\infty} \frac{(-1)^j(x_{j+1}-x_j)}{|j|^t\!\cdot\!|r_2|^j}=(-1)^t(r_2-1)b,\label{eq1}\\
& \ \lim_{j\to\infty} \frac{(-1)^j(x_{j+1}-x_j)}{j^s\!\cdot\!|r_1|^j}=(r_1-1)a.\label{eq2}
\end{align}
This shows that the sequence $(x_j)_{j\in\Z}$ cannot be monotone (in fact, this shows that it cannot be monotone when either $A\not\equiv0$ or $B\not\equiv0$); consequently, $f$ cannot be increasing. Thus $f$ is decreasing.

According to the above observation the expression $(-1)^j(x_{j+1}-x_j)$ has a constant sign when $j$ runs through $\Z$. Combining this fact with equations \eqref{eq1} and \eqref{eq2}, we conclude that $a$ and $(-1)^tb$ have the same sign. Further, since $f^2$ is increasing, the expression
\begin{multline*}
x_{2j+2}-x_{2j}=\big(A(2j+2)r_1^2-A(2j)\big)|r_1|^{2j}\\ +F(2j+2)-F(2j)+\big(B(2j+2)r_2^2-B(2j)\big)|r_2|^{2j}
\end{multline*}
also has a constant sign. Similarly, we have
\[ \lim_{j\to-\infty}\frac{x_{2j+2}-x_{2j}}{|2j|^t\!\cdot\!|r_2|^{2j}}=(-1)^t(r_2^2-1)b,\qquad\lim_{j\to\infty}\frac{x_{2j+2}-x_{2j}}{(2j)^s\!\cdot\!|r_1|^{2j}}=(r_1^2-1)a. \]
As a result, $a$ and $(-1)^tb$ are of opposite sign; a contradiction. Therefore, $A\equiv0$ or $B\equiv0$. Using Theorem \ref{thmrec} once again we conclude that the assertion holds for a~fixed $x\in I$. It remains to show that elimination of the root $r_1$ or $r_2$ does not depend on $x$.

Consider $y\in I$ and define a sequence $(y_j)_{j\in\Z}$ in the same way as we did for $x$. Suppose, for the sake of a contradiction, that $x_j=A(j)r_1^j+F(j)$ and $y_j=G(j)+B(j)r_2^j$ for $j\in\Z$ with non-zero polynomials $A$ and $B$ ($F$ and $G$ stand for the terms for which the roots $\lambda_1,\ldots,\lambda_p$ are responsible). As before, let $s$, $t$ be the degrees and $a$, $b$ be the leading coefficients of $A$ and $B$, respectively. Since $f$ monotonically decreases, the sequence $(x_j-y_j)_{j\in\Z}$ alternates in sign. Thus the expression
\[ (-1)^j(x_j-y_j)=A(j)|r_1|^j+(-1)^j\big(F(j)-G(j)\big)-B(j)|r_2|^j \]
has a constant sign. Futher, we have
\[ \lim_{j\to-\infty}\frac{(-1)^j(x_j-y_j)}{|j|^t\!\cdot\!|r_2|^j}=(-1)^{t+1}b,\qquad\lim_{j\to\infty}\frac{(-1)^j(x_j-y_j)}{j^s\!\cdot\!|r_1|^j}=a \]
which means that $a$ and $(-1)^{t+1}b$ have the same sign. Repeating this reasoning with the sequence $(x_{j+1}-y_j)_{j\in\Z}$ we conclude that $a$ and $(-1)^{t+1}b$ have opposite signs. The obtained contradiction ends the proof.
\end{proof}

\begin{remark}
Since the equation $2f^2(x)+5f(x)+2x=0$ is satisfied by $f(x)=-2x$ and $f(x)=-\frac12x$, in general, it cannot be decided which root from $r_1$ and $r_2$ may be eliminated. Therefore, Theorem \ref{thmlower} states that equation \eqref{eqmain} is actually equivalent to an alternative of two equations of lower order.
\end{remark}

\begin{remark}
It is worth mentioning that if $I=\R$, then $f$ is necessarily bijective (see \cite[Lemma 1]{zhang-zhang}). Therefore, the assumption of surjectivity in Theorems \ref{thmref} and \ref{thmlower} is satisfied automatically in this case.
\end{remark}

\begin{remark}
Using quoted results and Theorem \ref{thmlower} the order of equation \eqref{eqmain} can be essentially lowered in many important cases. However, some cases still remains open. For instance, it is unknown whether non-real roots may be eliminated from characteristic equation \eqref{eqchar} without any additional assumptions (cf. \cite[Section 6]{draga-morawiec} and \cite[Section 6]{zhang-zhang}).
\end{remark}

\subsection*{Acknowledgement}
The research was supported by the University of Silesia Mathematics Department (Iterative Functional Equations and Real Analysis program).


\begin{thebibliography}{9}
\bibitem{draga-morawiec} S. Draga, J. Morawiec, \textit{Reducing the polynomial-like iterative equations order and a~generalized Zolt\'an Boros' problem}, Aequationes Math. \textbf{90} (2016), 935--950.
\bibitem{jerri} A.J. Jerri, \textit{Linear Difference Equations with Discrete Transform Methods}, Springer 1996.
\bibitem{matkowski-zhang1} J. Matkowski, W. Zhang, \textit{Method of characteristic for functional equations in polynomial form}, Acta Math. Sinica \textbf{13} (1997), 421--432.
\bibitem{matkowski-zhang2} J. Matkowski, W. Zhang, \textit{Characteristic analysis for a polynomial-like iterative equation}, Chinese Sci. Bull. \textbf{43} (1998), 192--196.
\bibitem{nabeya} S. Nabeya, \textit{On the functional equation $f(p+qx+rf(x))=a+bx+cf(x)$}, Aequationes Math.~\textbf{11} (1974), 199--211.
\bibitem{yang-zhang} D. Yang, W. Zhang, \textit{Characteristic solutions of polynomial-like iterative equations}, Aequationes Math. \textbf{67} (2004), 80--105.
\bibitem{zhang-gong} P. Zhang, X. Gong, \textit{Continuous solutions of $3$-order iterative equation of linear dependence}, Adv. Difference Equ. 2014 (2014), 318.
\bibitem{zhang-xu-zhang} W. Zhang, B. Xu, W. Zhang, \textit{Global solutions for leading coefficient problem of polynomial-like iterative equations}, Results. Math. \textbf{63} (2013), 79--93.
\bibitem{zhang-zhang} W. Zhang, W. Zhang, \textit{On continuous solutions of $n$-th order polynomial-like iterative equations}, Publ. Math. Debrecen \textbf{76} (2010), 117--134.
\end{thebibliography}
\end{document}